\documentclass[a4paper,12pt]{amsart}
\usepackage{amsmath,amssymb}
\usepackage[ansinew]{inputenc}

 \newtheorem{thm}{Theorem}[section]
 \newtheorem{cor}[thm]{Corollary}
 \newtheorem{lem}[thm]{Lemma}
 \newtheorem{prop}[thm]{Proposition}
 \theoremstyle{definition}
 \newtheorem{defn}[thm]{Definition}
 \theoremstyle{remark}
 \newtheorem{rem}[thm]{Remark}
 \newtheorem{ex}{Example}
 \numberwithin{equation}{section}

\begin{document}
%
%
%
%
%
%
%
%
%

\newcommand{\oo}{{\scriptstyle{\mathcal{O}}}}
\newcommand{\OO}{{\mathcal {O}}}
\newcommand{\nat}{\mathbb N}
\newcommand{\weak}{{\curlyvee}}
\newcommand{\C}{{\mathbb C}}
\newcommand{\Q}{{\mathbb Q}}
\newcommand{\R}{{\mathbb R}}
\newcommand{\Z}{{\mathbb Z}}
\newcommand{\I}{{\mathbf I}}
\def\leq{\leqslant}
\def\geq{\geqslant}

\def\b{\blacksquare}

\newcommand{\codim}{codim}
\newcommand{\Sing}{Sing}
\newcommand{\Supp}{Supp}


\def\baselinestretch{1.1}

\title[\L ojasiewicz exponents and resolution of singularities]{\L ojasiewicz exponents and resolution of \\ singularities}
\author{C. Bivi\`a-Ausina}

\address{%
Institut Universitari de Matem\`atica Pura i Aplicada\\
Universitat Polit\`ecnica de Val\`encia\\
Cam\'i de Vera, s/n\\
46022 Val\`encia\\
Spain}

\email{carbivia@mat.upv.es}

\thanks{The first author was partially supported by DGICYT
Grant MTM2006--06027. The second author was partially supported by
DGICYT Grant MTM2006--10548.}
\author{S. Encinas}

\address{Departamento de Matem\'atica Aplicada\\
Universidad de Valladolid\\
Avda. Salamanca s/n\\
47014 Valladolid \\
Spain}

\email{sencinas@maf.uva.es}
\def\subjclassname{$2000$ Mathematics Subject Classification}
\subjclass[$2000$ Mathematics Subject Classification]{Primary 2S05, 14E15; Secondary 13H15}





\begin{abstract}
We show an effective method to compute the \L ojasiewicz exponent
of an arbitrary sheaf of ideals of $\OO_X$, where $X$ is a
non-singular scheme. This method is based on the algorithm of
resolution of singularities.
\end{abstract}

\maketitle
\section{Introduction}

Given an analytic function $f:(\C^n,0)\to (\C,0)$ with an isolated
singularity at the origin, the effective computation of the \L
ojasiewicz exponent $\mathcal L_0(f)$ of $f$ is a problem that has
been approached from both algebraic and analytic techniques (see
for instance \cite{Bivia2008prep}, \cite{Fukui1991},
\cite{Lenarcik1998} or \cite{Ploski1988}). This number is defined
as the infimum of those real numbers $\alpha>0$ such that
$$
\Vert x\Vert^\alpha\leq C \Vert \nabla f(x)\Vert,
$$
for some constant $C>0$ and all $x$ belonging to some open
neighbourhood of the origin in $\C^n$, where $\nabla f$ denotes
the gradient of $f$. One of the most significant applications of
$\mathcal L_0(f)$ is the result of B. Teissier \cite[p.\ 280]{Teissier1977}
stating that the degree of topological determinacy of $f$ is equal
to $[\mathcal L_0(f)]+1$, where $[a]$ denotes the integer part of
a number $a\in\R$. Let us denote by $j^rf$ the {\it $r$-jet} of $f$, that is, the sum of all terms
of the Taylor expansion of $f$ around the origin of degree $\leq r$. Then the degree of topological determinacy of $f$
is defined as the minimum of those $r\geq 1$ such that
for all $g\in\OO_n$ verifying that $j^rf=j^rg$, we have that $f$ and $g$ are topologically equivalent, that is,
there exists a germ of homeomorphism $\varphi:(\C^n,0)\to (\C^n,0)$ such that $f=g\circ \varphi$.


Let us denote by $\mathcal O_n$ the ring of analytic functions
$f:(\C^n,0)\to \C$. The definition of \L ojasiewicz exponent of
functions with an isolated singularity is extended naturally to
ideals of $\OO_n$ of finite colength. Let $I$ be an ideal of
$\OO_n$. In this article we apply the explicit construction of a
log-resolution of $I$ given in \cite{BravoEncinasVillamayor2005}
to compute effectively the \L ojasiewicz exponent $\mathcal
L_0(I)$ of $I$ provided that $I$ has finite colength. We consider
the problem of computing $\mathcal L_0(I)$ in a more general
setting, that is, we substitute $I$ by a sheaf of ideals in a
non-singular scheme.


As an application of the main result, we compute the \L ojasiewicz
exponent, and consequently the degree of topological determinacy,
of a function such that $\mathcal L_0(f)$ can not be computed by
means of the existing literature about this subject.



\section{Order functions}

In this section we recall some known facts concerning the integral
closure of ideals and its relation with reduced orders. We will
denote by \( R \) a Noetherian ring.

\begin{defn} \label{DefOrderFunc}
Let \( \bar{\mathbb{R}}_{0}=\{a\in\mathbb{R}\mid a\geq
0\}\cup\{\infty\} \) and let us consider a function \(
\rho:R\to\bar{\mathbb{R}}_{0} \). We say that $\rho$ is an {\it
order function} if the following conditions hold:
\begin{enumerate}
    \item[(i)]  \( \rho(f+g)\geq\min\{\rho(f),\rho(g)\} \), for all \(
    f,g\in R \).

    \item[(ii)]  \( \rho(fg)\geq \rho(f)+\rho(g) \), for all \( f,g\in R \).

    \item[(iii)]   \( \rho(0)=\infty \) and \( \rho(1)=0 \).
\end{enumerate}
\end{defn}

Let \( I\subseteq R \) be an ideal and let $f\in R$. It is well
known, and also easy to prove, that the function
\begin{equation*}
    \nu_{I}(f)=\sup\{m\in \mathbb{N}\mid f\in I^{m}\}
\end{equation*}
is an order function. Let \( J\subseteq R \) be an ideal and set
\begin{equation*}
    \nu_{I}(J)=\sup\{m\in \mathbb{N}\mid J\subseteq I^{m}\}.
\end{equation*}
If \( f_{1},\ldots,f_{s} \) are generators of \( J \), then
it can be checked that
\begin{equation*}
    \nu_{I}(J)=\min\{\nu_{I}(f_{1}),\ldots,\nu_{I}(f_{s})\}.
\end{equation*}

\begin{prop} \label{DefPropNuBar} \textnormal{
\cite{Rees1956}\cite[\S 0.2]{LejeuneTeissier1974}} Let \(
I\subseteq R \) be an ideal with \( I\neq R \). Then the sequence
\begin{equation*}
    \left\{\frac{\nu_{I}(f^{n})}{n}\right\}_{n=1}^{\infty}
\end{equation*}
has a limit in \( \bar{\mathbb{R}}_{0} \). Moreover the function \(
\bar{\nu}_{I}:R\to\bar{\mathbb{R}}_{0} \) defined by
\begin{equation*}
    \bar{\nu}_{I}(f)=\lim_{n\to\infty}\frac{\nu_{I}(f^{n})}{n}
\end{equation*}
is an order function.
\end{prop}

The number $\bar{\nu}_{I}(f)$ is called the {\it reduced order of
$f$ with respect to $I$}. It is proved in
\cite{LejeuneTeissier1974} that \(
\bar{\nu}_{I}(f)\in\mathbb{Q}_+\cup\{\infty\} \), for all $f\in
R$.
We will show this result using the existence of embedded
desingularization of schemes and log-resolution of ideals.

\begin{rem} \label{RemNuBarCrec}
The sequence \(
\left\{u_{n}=\frac{\nu_{I}(f^n)}{n}\right\}_{n=1}^{\infty} \) is
not an increasing sequence, in general. However, it is straightforward to see
that, for any positive integer $i\geq 2$, the subsequence \(
\{u_{i^{n}}\}_{n=1}^{\infty} \) is increasing, so that
\begin{equation*}
    \bar{\nu}_{I}(f)=\lim_{n\to\infty}\frac{\nu_{I}(f^{n})}{n}=
    \sup\left\{\frac{\nu_{I}(f^{n})}{n}\mid n\in\mathbb{N}\right\}
\end{equation*}
and \( n\bar{\nu}_{I}(f)\geq\nu_{I}(f^{n}) \) for all \( n \). In
particular \( \bar{\nu}_{I}(f)\geq\nu_{I}(f) \), for all $f\in R$.
\end{rem}

\begin{lem} \label{Lema_pq_barnu}
\textnormal{\cite[0.2.9]{LejeuneTeissier1974}} Let $I$ and $J$ be
ideals of $R$ and let $p,q$ be positive integers. Then
\begin{equation*}
    \bar{\nu}_{I^{p}}(J^{q})(x)=
    \frac{q}{p}\bar{\nu}_{I}(J).
\end{equation*}
\end{lem}

For an ideal \( I \) of \( R \) we will denote by \( \overline{I}
\) the integral closure of \( I \).

\begin{lem} \label{LemaBarNu1}
\textnormal{\cite[1.15]{LejeuneTeissier1974}\cite[p.\
138]{HunekeSwanson2006}} Let $R$ be a Noetherian ring and let
$I,J$ be ideals of $R$. If \(J\subseteq\overline I \), then \(
\overline\nu_{I}(J)\geq 1 \).
\end{lem}

\begin{defn} \label{DefOrdMu}
Let \( I\subseteq R \) be an ideal. We define the function \(
\mu_{I}:R\to\bar{\mathbb{R}}_{0} \) as
\begin{equation*}
    \mu_{I}(f)=\sup\left\{\frac{p}{q}\in\Q_+\mid f^{q}\in
    \overline{I^{p}}\right\}.
\end{equation*}
\end{defn}

As a consequence of \cite[Proposition 10.5.2]{HunekeSwanson2006}
(see also \cite[\S 4.2]{LejeuneTeissier1974}) the set of rational
numbers involved in Definition~\ref{DefOrdMu} does not depend on
the representatives \( p,q \) of the rational number \(
\frac{p}{q} \).

Let us consider the graded ring \( R[T] \), with the usual
graduation on \( T \). Let \( R[IT]\subseteq R[T] \) be the
subring \( R[IT]=\oplus_{n}I^{n}T^{n} \). Let \( f\in R \), we
have that \( f\in\bar{I} \) if and only if the homogeneous element
\( fT\in R[T] \) is in the integral closure of the ring \( R[IT]
\) in \( R[T] \). It is well known (see, for instance, \cite[p.\
95]{HunekeSwanson2006}) that this integral closure is
\begin{equation*}
    \overline{R[IT]}=\bigoplus_{n}\overline{I^{n}}T^{n}
    \subseteq R[T].
\end{equation*}

\begin{lem} \label{LemaPotencias}
If \( f^{q}\in\overline{I} \) and \( g^{q}\in\overline{I} \) then \(
(f+g)^{q}\in\overline{I} \).
\end{lem}

\begin{proof}
By assumption \( f^{q}T \) and \( g^{q}T \) are integral over \(
R[IT] \). We observe that the ring extension \( R[T]\subseteq
R[T^{\frac{1}{q}}] \) is finite. Then \( fT^{\frac{1}{q}} \) and
\( gT^{\frac{1}{q}} \) are integral over \( R[IT]\subseteq
R[T^{\frac{1}{q}}] \). Therefore \( (f+g)T^{\frac{1}{q}} \) is
integral over \( R[IT] \). Thus \( (f+g)^{q}T \) is integral over
\( R[IT] \) and we conclude that \( (f+g)^{q}\in\overline{I} \).
\end{proof}

\begin{prop}
Let $I$ be an ideal of $R$. Then \( \mu_{I} \) is an order
function.
\end{prop}

\begin{proof}
The fact that $\mu_I$ satisfies condition (i) of Definition
\ref{DefOrderFunc} follows as a direct application of Lemma
\ref{LemaPotencias}. Conditions (ii) and (iii) follow easily from
the definition of $\mu_I$.
\end{proof}

\section{Resolution of singularities and integral closure.}

In this section, \( X \) will denote an integral separated scheme
of finite type over a field \( k \), where the characteristic of
\( k \) is zero.

If \( \mathcal{I}\subseteq \OO_{X} \) is a sheaf of ideals then
the integral closure \( \overline{\mathcal{I}} \) is a sheaf of
ideals such that for every point $x\in X$, the ideal \(
\overline{\mathcal{I}}_{x} \) is the integral closure of \(
\mathcal{I}_{x}\subseteq\OO_{X,x} \).

The next result is well known and its proof can be found, for
instance, in \cite[p.\ 133]{HunekeSwanson2006}.

\begin{lem}
\label{LemaValorIC} Let \( R \) be a Noetherian domain. Denote by \(
K \) the field of fractions of \( R \). Let \( I\subseteq R \) be
an ideal. For every valuation ring \( R_{v}\subseteq K \) set \(
I_{v}=(IR_{v})\cap R \). Then the integral closure of \( I \) is
\( \overline{I}=\bigcap_{v}I_{v} \), where the intersection ranges
on all valuation rings in \( K \) with center in \( R \).
\end{lem}

\begin{prop} \label{PropIntClProper}
Let \( \varphi:X'\to X \) be a proper birational morphism and let
\( \mathcal{I}\subseteq\OO_{X} \) be a sheaf of ideals. Then \(
\overline{\mathcal{I}}=(\overline{\mathcal{I}\OO_{X'}})\cap\OO_{X}
\).
\end{prop}

\begin{proof}
It is a consequence of Lemma \ref{LemaValorIC} and the Valuative
Criterion of Properness \cite[Theorem 4.7, \S II]{Hartshorne}.
\end{proof}

\begin{defn} \label{DefResol}
A \emph{desingularization} of \( X \) is a proper birational morphism \(
\varphi:X'\to X \) such that
\begin{enumerate}
    \item[(i)]  \( X' \) is non-singular;

    \item[(ii)] the morphism \( \varphi \) is an isomorphism outside the singular locus
    of \( X \). That is, if \( U=X\setminus\Sing(X) \) and \(
    U'=\varphi^{-1}(U) \), then \( U'\cong U \) via \( \varphi \).
\end{enumerate}
Assume that \( X\subseteq W \), where \( W \) is a non-singular
scheme. An \emph{embedded desingularization} of \( X\subseteq W \)
is a proper birational morphism \( \Pi:W'\to W \) such that
\begin{enumerate}
    \item[(i)]  \( W' \) is non-singular;

    \item[(ii)] the morphism \( \Pi \) is an isomorphism outside the singular locus
    of \( X \). That is, if \( U=W\setminus\Sing(X) \) and \(
    U'=\Pi^{-1}(U) \), then \( U'\cong U \) via \( \Pi \);

    \item[(iii)] \( W'\setminus U' \) is a simple divisor with normal
    crossings: \( W'\setminus U'=H_{1}\cup\cdots\cup H_{r} \);

    \item[(iv)] if \( X'\subseteq W' \) is the strict transform of \( X \)
    in \( W' \) then \( X' \) is non-singular and has only normal
    crossings with the divisor \( W'\setminus U' \).
\end{enumerate}
\end{defn}

\begin{defn} \label{DefLogResol}
Let \( W \) be non-singular scheme. A {\it log-resolution} of an
ideal \( \mathcal{I}\subseteq\OO_{W} \) is a proper birational
morphism \( \Pi:W'\to W \) such that
\begin{enumerate}
    \item[(i)]  \( W' \) is non-singular,

    \item[(ii)]  \( \Pi \) is an isomorphism outside the support of \( \mathcal{I} \).
    If \( U=W\setminus\Supp(\mathcal{I}) \) and \( U'=\Pi^{-1}(U) \) then \(
    U'\cong U \) via \( \Pi \).

    \item[(iii)] \( W'\setminus U' \) is a simple divisor with normal crossings: \(
    W'\setminus U'=H_{1}\cup\cdots\cup H_{r} \).

    \item[(iv)] The total transform of \( \mathcal{I} \) in \( W' \) is a
    monomial with support in $W'\setminus U'$
    \begin{equation}\label{totaltransfI}
        \mathcal{I}\OO_{W'}=\I(H_{1})^{a_{1}}\cdots
        \I(H_{r})^{a_{r}}.
    \end{equation}
\end{enumerate}
\end{defn}

\begin{rem}
It was proved by Hironaka in \cite{Hironaka1964} that embedded
desingularizations and log-resolutions do exist without
restriction on the dimension of schemes over a field of
characteristic zero. In fact, Hironaka proved
that the morphism \( \Pi \) may be obtained as a sequence of
blowing-ups along regular centers.

The proof in \cite{Hironaka1964} is existential. Constructive
proofs may be found in \cite{Villamayor1989} and also in
\cite{BierstoneMilman1991}. If the characteristic of the ground
field \( k \) is positive, then resolution of singularities is an
open problem for general dimension. The reader may found more
details in \cite{Hauser2003}. We refer to
\cite{BravoEncinasVillamayor2005} for constructive proofs of
embedded desingularization of schemes, log-resolution of ideals
and (non-embedded) desingularization of schemes.

Algorithms implementing resolution of singularities (in
characteristic zero) in the computer are available for explicit
computations. We will use the implementation of
\cite{BodnarSchicho2000b} available at
\begin{center}
    \texttt{http://www.risc.uni-linz.ac.at/projects/basic/adjoints/blowup}
\end{center}
and implemented in Singular \cite{SINGULAR3} and Maple.
There is another implementation of resolution of singularities in
\cite{FruhbisPfister2004} also implemented in Singular.
\end{rem}

\begin{prop}
Let us consider a log-resolution of \( \mathcal{I}\subseteq\OO_{W}
\), as in Definition \textnormal{\ref{DefLogResol}}. Then
\begin{equation*}
    \overline{\mathcal{I}^{m}}=
    \I(H_{1})^{ma_{1}}\cdots \I(H_{r})^{ma_{r}} \cap
    \OO_{W},
\end{equation*}
for any integer $m\geq 1$.
\end{prop}

\begin{proof}
It is a consequence of Proposition \ref{PropIntClProper} and the
fact that locally principal ideals are integrally closed.
\end{proof}

\section{The reduced order of a sheaf and \L ojasiewicz exponents}

As in the previous section, here \( X \) will denote an integral
separated scheme of finite type over a field \( k \).

\begin{defn}
Let \( \mathcal{I},\mathcal{J}\subseteq\OO_{X} \) be two sheaves
of ideals. We define two functions \(
\bar{\nu}_{\mathcal{I}}(\mathcal{J}):X\to\bar{\mathbb{R}}_{0} \)
and \( \mu_{\mathcal{I}}(\mathcal{J}):X\to\bar{\mathbb{R}}_{0} \)
as follows
\begin{equation*}
    \bar{\nu}_{\mathcal{I}}(\mathcal{J})(x)=
    \bar{\nu}_{\mathcal{I}_{x}}(\mathcal{J}_{x})=
    \inf_{f\in \mathcal{J}_{x}}\bar{\nu}_{\mathcal{I}_{x}}(f),
    \qquad
    \mu_{\mathcal{I}}(\mathcal{J})(x)=
    \mu_{\mathcal{I}_{x}}(\mathcal{J}_{x})=
    \inf_{f\in \mathcal{J}_{x}}\mu_{\mathcal{I}_{x}}(f),
\end{equation*}
for all $x\in X$.
\end{defn}

We say that a function $\mu:X\to \R\cup\{\infty\}$ is {\it
lower-semicontinuous} if for any $\alpha\in\R$, the set
$F_{\alpha}=\{x\in X\mid \mu(x)\leq \alpha\}$ is closed.
Analogously, we say that $\mu$ is {\it upper-semicontinuous} when
the set $G_{\alpha}=\{x\in X\mid \mu(x)\geq \alpha\}$ is closed,
for all $\alpha\in\R$.

\begin{lem} \label{LemaLambda}
Assume that \( X \) is non-singular and that \( H_{1},\ldots,H_{r}
\) are non-singular irreducible hypersurfaces having only normal
crossings. Let \(
\lambda_{1},\ldots,\lambda_{r}\in\bar{\mathbb{R}}_{0} \) and let
$\bar{\mathbb N}=\mathbb N\cup\{\infty\}$. Let us consider the
function \( \lambda_{i}:X\to\bar{\mathbb{N}} \) given by
\begin{equation*}
\lambda_{i}(x)=
\begin{cases}
\lambda_{i}, & \textnormal{if $x\in H_{i}$,}\\
 \infty, &\textnormal{otherwise}.
\end{cases}
\end{equation*}
Then the function $\lambda:X\to \bar{\mathbb N}$ defined by \(
\lambda=\min\{\lambda_{i}\mid i=1,\ldots,N\} \) is
lower-semicontinuous.
\end{lem}

\begin{proof}
Let \( \alpha\in\bar{\mathbb{R}}_{0} \) and let us consider the
set \( F_{\alpha}=\{x\in X \mid \lambda(x)\leq \alpha\} \). We
observe that \( F_{\alpha} \) is the union of the hypersurfaces
$H_i$ such that \( \lambda_{i}\leq\alpha \). Therefore \(
F_{\alpha}\) is closed and the result follows.
\end{proof}

Let \( \mathcal{I},\mathcal{J}\subseteq\OO_{X} \) be two sheaves
of ideals. Let \( \Pi':X''\to X \) be a desingularization of \( X
\) (in the sense of Definition \textnormal{\ref{DefResol}}) and
let \( \Pi'':X'\to X'' \) be a log-resolution of \(
\mathcal{I}\OO_{X''} \) (as in Definition
\textnormal{\ref{DefLogResol}}), so that
\begin{equation}\label{logres}
    \mathcal{I}\OO_{X'}=\I(H_{1})^{a_{1}}\cdots \I(H_{r})^{a_{r}},
\end{equation}
for some positive integers $a_1,\dots, a_r$. The total transform
\( \mathcal{J}\OO_{X'} \) can be expressed as
\begin{equation}\label{Jprima}
    \mathcal{J}\OO_{X'}=
    \I(H_{1})^{b_{1}}\cdots \I(H_{r})^{b_{r}}\mathcal{J}',
\end{equation}
where \( \mathcal{J}'\subseteq\OO_{X'} \) and \(
\mathcal{J}'\not\subseteq \I(H_{i}) \), for all \( i=1,\ldots,r
\).

\begin{prop} \label{ThFuncionMu} In the setup described above, let us consider the function \(
\lambda=\min\{\frac{b_{i}}{a_{i}}\mid i=1,\ldots,r\}\). Then
\begin{equation*}
    \mu_{\mathcal{I}}(\mathcal{J})(x)=
    \min\big\{\lambda(x')\mid x'\in\Pi^{-1}(x)\big\},
\end{equation*}
for all \( x\in X \), and the function \( \mu_{\mathcal
I}(\mathcal J) \) is lower-semicontinuous.
\end{prop}

\begin{proof} Let $p,q$ be positive integers. We observe that \(
(\mathcal{J}^{q})_{x}\subseteq (\overline{\mathcal{I}^{p}})_{x} \)
if and only if \( (\mathcal{J}^{q}\OO_{X'})_{x'}\subseteq
(\mathcal{I}^{p}\OO_{X'})_{x'} \), for all \( x'\in\Pi^{-1}(x) \).
Moreover, according to (\ref{logres}) and (\ref{Jprima}), we
have the following equivalences:
\begin{align*}
    (\mathcal{J}^{q}\OO_{X'})_{x'}\subseteq (\mathcal{I}^{p}\OO_{X'})_{x'}
     & \Longleftrightarrow
     (\I(H_{1})^{qb_{1}}\cdots
     \I(H_{r})^{qb_{r}}{\mathcal J'}^{q})_{x'}\subseteq
     (\I(H_{1})^{pa_{1}}\cdots \I(H_{r})^{pa_{r}})_{x'} \\
     & \Longleftrightarrow
     \left(\frac{b_{i}}{a_{i}}\right)(x')\geq \frac{p}{q}, \ i=1,\ldots,r \\
     & \Longleftrightarrow
     \lambda(x')\geq \frac{p}{q}.
\end{align*}
Hence
\begin{align*}
    \mu_{\mathcal I}(\mathcal J)(x)\geq \frac{p}{q} & \Longleftrightarrow
    \lambda(x')\geq \frac{p}{q}, \quad \textnormal{for all $x'\in\Pi^{-1}(x)$},
\end{align*}
and we have \( \mu_{\mathcal I}(\mathcal
J)(x)=\min\{\lambda(x')\mid x'\in\Pi^{-1}(x)\} \).

The lower-semicontinuity of $\mu_\mathcal I(\mathcal J)$ follows
from the properness of $\Pi$.
\end{proof}

As an immediate consequence of the previous theorem we obtain the
following result.

\begin{cor}
The value \( \mu_{\mathcal{I}}(\mathcal{J})(x) \) is rational, for
every \( x\in X \).
\end{cor}

\begin{thm} Let $\mathcal I,\mathcal J\subseteq \mathcal O_X$
be two sheaves of ideals. Then the functions
$\bar{\nu}_{\mathcal{I}}(\mathcal{J})$ and
$\mu_{\mathcal{I}}(\mathcal{J})$ are equal.
\end{thm}

\begin{proof}
We use the same notation as in Proposition \ref{ThFuncionMu}. Let
us fix a point \( x\in X \). First we prove that \(
\mu_{\mathcal{I}}(\mathcal{J})\geq
\bar{\nu}_{\mathcal{I}}(\mathcal{J}) \).

Set \( c_{n}=\nu_{\mathcal{I}}(\mathcal{J}^{n})(x) \), for all
$n\geq 1$. We observe that
\begin{equation*}
    \bar{\nu}_{\mathcal{I}}(\mathcal{J})(x)=
    \sup_{n\in\mathbb{N}}\frac{c_n}{n}.
\end{equation*}
By definition we have \( \mathcal{J}^{n}\subseteq
\mathcal{I}^{c_{n}}\subseteq\overline{\mathcal{I}^{c_{n}}} \),
which implies that
\begin{equation*}
    \mu_{\mathcal{I}}(\mathcal{J})(x)\geq \frac{c_{n}}{n},\quad
    \textnormal{for all $n\in\mathbb N$}.
\end{equation*}
Therefore
\begin{equation*}
    \mu_{\mathcal{I}}(\mathcal{J})(x)\geq
    \bar{\nu}_{\mathcal{I}}(\mathcal{J}).
\end{equation*}
Conversely, set \( \frac{p}{q}=\mu_{\mathcal{I}}(\mathcal{J})(x)
\). This implies that \(
\mathcal{J}^{q}_{x}\subseteq\overline{{\mathcal I}^p}_x \). By
Lemma~\ref{LemaBarNu1} we have that \(
\bar{\nu}_{\mathcal{I}^{p}}(\mathcal{J}^{q})(x)\geq 1 \) and from
Lemma~\ref{Lema_pq_barnu} we obtain \(
\bar{\nu}_{\mathcal{I}}(\mathcal{J})(x)\geq \frac{p}{q} \).
\end{proof}

\begin{cor}\label{nuesracional}
The value \( \bar{\nu}_{\mathcal{I}}(\mathcal{J})(x) \) is
rational, for every \( x\in X \).
\end{cor}

\begin{defn}
Let \( X \) be a scheme as above with structure of complex
variety. Let \( \mathcal{I}\subseteq\OO_{X} \) be a coherent sheaf
of ideals and \( K\subseteq X \) be a compact set. Let \(
f\in\Gamma(X,\OO_{X}) \). The {\it {\L}ojasiewicz exponent of $f$
with respect to \( \mathcal{I} \) at \( K \)}, denoted by \(
\theta_{K}(f,\mathcal{I}) \), is defined as the infimum of those
$\theta\in\mathbb{R_+}$ such that there exists an open set $U\subseteq\C^n$ such that $K\subseteq U$
and a constant $C\geq 0$ such that
$$
|f(x)|^{\theta}\leq C\cdot\sup_{g\in\Gamma(U,\mathcal{I})}|g(x)|,
$$
for all $x\in U$.

If \( \mathcal{J}\subseteq\OO_{X} \) is a sheaf of ideals, then
\begin{equation*}
    \theta_{K}(\mathcal{J},\mathcal{I})=
    \sup_{f\in\Gamma(X,\mathcal{J})}\theta_{K}(f,\mathcal{I}).
\end{equation*}
\end{defn}

\begin{thm}\label{invers}
\textnormal{\cite[6.3]{LejeuneTeissier1974}} Under the hypothesis
of the previous definition we have
\begin{equation*}
    \theta_{K}(\mathcal{J},\mathcal{I})=
    \frac{1}{\bar{\nu}_{\mathcal{I}}(\mathcal{J})(K)},
\end{equation*}
where \( \bar{\nu}_{\mathcal{I}}(\mathcal{J})(K)=
\min\{\bar{\nu}_{\mathcal{I}}(\mathcal{J})(x)\mid x\in K\} \).
\end{thm}

As a direct consequence of the previous theorem and of Corollary
\ref{nuesracional} we obtain that the {\L}ojasiewicz exponent $
\theta_{K}(\mathcal{J},\mathcal{I})$ is a rational number.

\begin{defn}
Let $\mathcal I,\mathcal J\subseteq \mathcal O_X$ be two sheaves
of ideals. We define the function \(
\theta(\mathcal{J},\mathcal{I}):X\to\mathbb{Q} \) as follows:
\begin{equation*}
    \theta(\mathcal{J},\mathcal{I})(x)=
    \theta_{\{x\}}(\mathcal{J},\mathcal{I}),
\end{equation*}
for all $x\in X$.
\end{defn}

From Proposition \ref{ThFuncionMu} and Theorem \ref{invers} we
obtain that the function \( \theta(\mathcal{J},\mathcal{I}):X\to
\Q\) is upper-semicontinuous.

\section{Computation of {\L}ojasiewicz exponents for isolated
singularities.}

Let $W$ be an scheme with structure a regular analytic variety.
Let \( \mathcal{I} \) be a sheaf of ideals in \( \OO_{W} \) such
that \( \Supp(\mathcal{I})=\{x\} \), where \( x\in W \).
We define
the {\it {\L}ojasiewicz exponent of $\mathcal I$ at \( x \)} as \(
\mathcal{L}_{x}(\mathcal{I})=\theta(\mathcal{J},\mathcal{I})(x) \)
where \( \mathcal{J} \) is the sheaf of ideals
\begin{equation*}
    \mathcal{J}_{y}=\left\{
    \begin{array}{lll}
        \mathfrak{m}_{x} & \text{if} & y=x  \\
        1 & \text{if} & y\neq x.
    \end{array}
    \right.
\end{equation*}

\begin{thm}
The \L ojasiewicz exponent of \( \mathcal{I} \) is determined by
the total transform of \( \mathfrak{m}_{x}\) via the
log-resolution of \( \mathcal{I} \).
\end{thm}

\begin{proof}
Let us consider a log-resolution of \( \mathcal{I} \) as in
Definition \ref{DefLogResol}. The morphism \( W'\to W \) is a
sequence of blowing-ups along regular centers:
\begin{equation*}
    W=W_{0}\longleftarrow W_{1}\longleftarrow \cdots \longleftarrow
    W_{r}=W'.
\end{equation*}
We observe that the first blowing-up must have \(
\Supp(\mathcal{I}) \) as center. Therefore \(
\mathcal{J}\OO_{W_1}=\mathfrak{m}_{x}\OO_{W_{1}}=I(H_{1}) \) and
the total transform of \( \mathfrak{m}_{x} \) is a monomial, that
is
\begin{equation*}
    \mathfrak{m}_{x}\OO_{W'}=\I(H_{1})^{b_{1}}\cdots
    \I(H_{r})^{b_{r}},
\end{equation*}
for some positive integers $b_1,\dots, b_r$.

Let us suppose that the total transform of $\mathcal I$ in $W'$ is
written as in (\ref{totaltransfI}). Then, we obtain the following
equivalences:
\begin{align*}
    \mathfrak{m}_{x}^{p}\subseteq \overline{\mathcal{I}^{q}} \quad
    \Longleftrightarrow & \quad
    \I(H_{1})^{pb_{1}}\cdots \I(H_{r})^{pb_{r}} \subseteq     \I(H_{1})^{qa_{1}}\cdots \I(H_{r})^{qa_{r}} \\
    \Longleftrightarrow & \quad pb_{i}\geq qa_{i}, \quad
    i=1,\ldots,r  \\
    \Longleftrightarrow & \quad \frac{p}{q}\geq \frac{a_{i}}{b_{i}},
    \quad i=1,\ldots,r.
\end{align*}

Then, we conclude that
\begin{equation}\label{calculodeL}
\mathcal L_{_{x}}(\mathcal{I})=\max\left\{\dfrac{a_{i}}{b_{i}},\
i=1,\ldots,r\right\}.
\end{equation}
\end{proof}

By (\ref{calculodeL}), the problem of computing $\mathcal
L_{_{x}}(\mathcal{I})$ reduces to determine the integers $a_i,
b_i$, for $i=1,\dots, r$, which in turn, come from determining the
total transform of \( \mathfrak{m}_{x}\) via the log-resolution of
\( \mathcal{I} \). Next we expose some examples in the ring
$\OO_n$ of holomorphic gems $f:(\C^n,0)\to \C$.

\begin{ex} Let us consider the ideal $I$ of $\OO_3$ generated by the polynomials
\begin{align*}
g_1&=x^4+xyz+y^4\\
g_2&=xy^2z\\
g_3&=y^5+z^5.
\end{align*}

Then, applying relation (\ref{calculodeL}), it follows that
$\mathcal L_0(I)=5+\frac 56$. Let us denote by $e(I)$ the Samuel
multiplicity of $I$. The same value for $\mathcal L_0(I)$ is
obtained by following the approach explained in Section 4 of
\cite{Bivia2008prep}, since $e(I)$ equals the Rees mixed
multiplicity of the ideals $I_1=\langle x^4, xyz, y^4 \rangle$,
$I_2=\langle xy^2z \rangle$ and $I_3=\langle y^5, z^5\rangle$,
which is equal to $80$.
\end{ex}

\begin{ex} Let us consider the function $f\in \OO_3$ given by
$f(x,y,z)=y^6+z^4+x(x-3z)^2$ and let us denote by $\mu(f)$ the
Milnor number of $f$. We observe that $f$ is a Newton degenerate
function in the sense of \cite{Kouchnirenko1976}. Moreover
$\mu(f)=25$, whereas the Newton number of the Newton polyhedron of
$f$ is equal to $20$. Therefore, the \L ojasiewicz exponent of $f$
can not be computed using the technique explained in
\cite{Bivia2008prep} via mixed multiplicities of monomial ideals.

Using relation (\ref{calculodeL}) we obtain
$$
\mathcal L_0(\nabla f)=5.
$$

Therefore, by virtue of \cite{Teissier1977}, the degree of
topological determinacy of $f$ is given by
$$
[\mathcal L_0(\nabla f)]+1=6.
$$
\end{ex}

\end{document}